\documentclass{amsart}

\usepackage{amssymb} \usepackage{amsfonts} \usepackage{amsmath}
\usepackage{amsthm} \usepackage{epsfig} 
\usepackage{color}
\usepackage{amscd}
\usepackage[all]{xy}
\usepackage{graphicx}
\usepackage{pinlabel}
\usepackage{url}

\usepackage{pgf,tikz}
\usetikzlibrary{arrows}

\newtheorem{lemma}{Lemma}[section]
\newtheorem{teo}[lemma]{Theorem}
\newtheorem{prop}[lemma]{Proposition}
\newtheorem{cor}[lemma]{Corollary}

\newtheorem{claim}[lemma]{Claim} 

\theoremstyle{definition}
\newtheorem{defn}[lemma]{Definition}

\newtheorem{quest}[lemma]{Question}

\newtheorem{example}[lemma]{Example}

\theoremstyle{remark}
\newtheorem{rem}[lemma]{Remark} 

\newcommand{\matr} [4] {\big({\tiny\begin{array}{@{}c@{\ }c@{}} #1 & #2 \\ #3 & #4 \\ \end{array}} \big)}

\newcommand{\Iso}{{\rm Isom}}

\newcommand{\matZ} {\ensuremath {\mathbb{Z}}}

\newcommand{\matH} {\ensuremath {\mathbb{H}}}

\newcommand{\matCP} {\ensuremath {\mathbb{CP}}}

\newcommand{\Or}{\ensuremath {\rm O}}

\author{Bruno Martelli}
\address{Dipartimento di Matematica, Largo Pontecorvo 5, 56127 Pisa, Italy}
\email{martelli at dm dot unipi dot it}

\author{Stefano Riolo}
\address{Institut de math\'ematiques, Rue Emile-Argand 11, 2000 Neuch\^atel, Switzerland}
\email{stefano dot riolo at unine dot ch}

\author{Leone Slavich}
\address{Dipartimento di Matematica, Largo Pontecorvo 5, 56127 Pisa, Italy}
\email{leone dot slavich at gmail dot com}

\title[Compact hyperbolic manifolds without spin structures]{Compact hyperbolic manifolds \\ without spin structures}

\begin{document}

\begin{abstract}
We exhibit the first examples of compact orientable hyperbolic manifolds that do not have any spin structure. We show that such manifolds exist in all dimensions $n \ge 4$. 

The core of the argument is the construction of a compact oriented hyperbolic 4-manifold $M$ that contains a surface $S$ of genus 3 with self-intersection 1. The 4-manifold $M$ has an odd intersection form and is hence not spin. It is built by carefully assembling some right-angled 120-cells along a pattern inspired by the minimum trisection of $\matCP^2$. 

The manifold $M$ is also the first example of a compact orientable hyperbolic 4-manifold satisfying any of these conditions:
\begin{itemize}
\item $H_2(M,\matZ)$ is not generated by geodesically immersed surfaces.
\item There is a covering $\tilde M$ that is a non-trivial bundle over a compact surface.
\end{itemize}

\end{abstract}

\maketitle
\section{Introduction}

We prove here the following theorem.

\begin{teo} \label{main:teo}
There are compact orientable hyperbolic manifolds that do not admit any spin structure, in all dimensions $n\geq 4$. 
\end{teo}

We briefly describe the context. Every manifold is connected and without boundary in this introduction, unless otherwise stated.
Let $M$ be a smooth compact $n$-manifold. The following chain of implications is well-known:

\begin{center}
$M$ is parallelisable $\Longrightarrow$ $M$ is stably parallelisable $\Longrightarrow$ $M$ is almost parallelisable \\
$\Longrightarrow$ $w_i(M) = 0\  \forall i \geq 1$ $\Longrightarrow$ $M$ is spin $\Longrightarrow$ $M$ is orientable.
\end{center}

We recall the terminology, that is standard. The manifold $M$ is \emph{parallelisable} if its tangent bundle $TM$ is trivial; it is \emph{stably parallelisable} if the Whitney sum $TM \oplus \varepsilon^k$ is trivial for some $k$ (here $\varepsilon^k$ is the rank-$k$ trivial bundle over $M$); it is \emph{almost parallelisable} if $M \setminus \{\rm point\}$ is parallelisable; the symbol $w_i \in H^i(M,\matZ/_{2\matZ})$ denotes the $i$-th Stiefel-Whitney class of $M$; finally, $M$ is \emph{spin} if it admits a spin structure, and this holds precisely when $w_1 = w_2 = 0$, see \cite{H}. We remark that $w_1=0$ is equivalent to $M$ being orientable.

Compact orientable surfaces are stably parallelisable and compact orientable 3-manifolds are parallelisable (see \cite{BL} for a collection of elementary proofs). Things become more exciting in dimension 4, where everything depends on whether some appropriate characteristic classes vanish or not. Let $\chi$ and $\sigma$ be the Euler characteristic and the signature of a compact oriented 4-manifold $M$. The following holds (see for instance \cite{CS, Sc}):

\begin{align*}
M {\rm\ is\ parallelisable} & \ \Longleftrightarrow \ \chi = \sigma = 0\ \mbox{and}\ w_1 = w_2 = 0, \\
M {\rm\ is\ stably\ parallelisable} & \ \Longleftrightarrow\ \sigma = 0\ \mbox{and}\  w_1 = w_2 = 0, \\
M {\rm\ is\ almost\ parallelisable} & \ \Longleftrightarrow M {\rm \ is \ spin}\ \Longleftrightarrow \ w_1 = w_2 = 0. 
\end{align*}

We are interested here in compact hyperbolic manifolds. A manifold is \emph{virtually} P if it has a finite-sheeted cover that is P, where P is some property. We recall a theorem proved by Sullivan \cite{Su} in 1975, using previous work with Deligne \cite{DS}.

\begin{teo} [Deligne -- Sullivan]
Every compact hyperbolic $n$-manifold $M$ is virtually stably parallelisable.
\end{teo}
That is, the manifold $M$ has a finite-sheeted cover $\tilde M$ that is stably parallelisable. In other words, the tangent bundle of every compact hyperbolic manifold $M$ becomes trivial after first taking a finite cover and then adding a trivial bundle. This implies that every compact hyperbolic $n$-manifold $M$ is virtually spin.

The Deligne -- Sullivan Theorem shows in particular that there are plenty of stably parallelisable compact orientable hyperbolic manifolds in all dimensions. On the other hand, at the time of writing this paper, it seems unknown whether there exists any compact orientable hyperbolic manifold, in any dimension $n\geq4$,
that is \emph{not} stably parallelisable. We answer to this question in the affirmative for all $n\geq 4$ here in Theorem \ref{main:teo}, where we state the stronger assertion that there are non-spin compact orientable hyperbolic manifolds in all dimensions $n\geq 4$. 

We note that the first non-spin compact orientable flat manifolds were discovered by Auslander and Szczarba \cite{AS} in 1962. These exist in every dimension $n\geq 4$. Every compact flat manifold is virtually parallelisable thanks to Bieberbach's Theorem. In even dimensions a 
complete finite-volume hyperbolic $M$ is never parallelisable because $\chi(M)\neq 0$ by the generalised Gauss -- Bonnet theorem.

Using non-spin flat 4-manifolds as cusp sections, Long and Reid have recently constructed some non-spin finite-volume cusped orientable hyperbolic $n$-manifolds for all $n\geq 5$ in \cite{LR}. The paper also contains a nice short proof of the virtual spinness for finite-volume hyperbolic manifolds, together with an effective bound on the covering degree for many arithmetic manifolds of simplest type. 

\subsection*{Outline of the proof.} 
The core of the proof of Theorem \ref{main:teo} is the construction of a compact oriented hyperbolic 4-manifold $M$ with odd intersection form.

A compact oriented hyperbolic 4-manifold $M$ has $\chi(M)>0$ and $\sigma(M)=0$. Therefore $M$ is never parallelisable, and it is stably parallelisable if and only if it is spin. The manifold $M$ may have only two possible intersection forms up to equivalence over $\mathbb{Z}$: either even $\oplus_m\matr 0110$ or odd $\oplus_m(1) \oplus_m(-1)$. A spin 4-manifold $M$ must have an even intersection form, and the converse holds if $H_1(M,\matZ)$ has no 2-torsion.

The parity of the intersection form of compact hyperbolic 4-manifolds has been determined only in very few cases. The Davis manifold is even and hence spin because its first homology group is torsion-free \cite{Da, RT}, see also \cite{RRT}. More recently, the orientable small covers of the right-angled 120-cell have been classified: they are 56 and they all have even intersection form \cite{MZ}. We are not aware of any other compact oriented hyperbolic 4-manifold whose intersection form has been computed. See \cite{Ma} for a survey on finite-volume hyperbolic 4-manifolds.

We note the following fact.

\begin{prop}
Let $M$ be an orientable hyperbolic 4-manifold. If $H_2(M, \matZ)$ is generated by immersed totally geodesic surfaces, the intersection form of $M$ is even.
\end{prop}
\begin{proof}
If $S\subset M$ is totally geodesic and embedded, the normal bundle has a flat connection and is hence trivial. Therefore we have $S\cdot S = 0$. If $S$ is only immersed, its normal bundle is again trivial for the same reason. By desingularising we deduce that $S\cdot S$ is even.

If $H_2(M)$ is generated by totally geodesic immersed surfaces $S_1,\ldots, S_k$, then $S_i \cdot S_i$ even for all $i$ implies that the intersection form is even.
\end{proof}

For instance, the second integral homology group of the Davis manifold has rank $72$ and is generated by 72 totally geodesic embedded surfaces of genus 2, as proved in \cite{RT}. Therefore the Davis manifold has an even intersection form.

How can we construct a compact hyperbolic 4-manifold with odd intersection form? The only techniques we know to build hyperbolic 4-manifolds essentially use either Coxeter polytopes or arithmetic groups, and both procedures typically produce a lot of totally geodesic immersed submanifolds, so some care is needed. We prove here the following.

\begin{teo} \label{odd:teo}
There is a compact oriented arithmetic hyperbolic 4-manifold $M$ that contains a $\pi_1$-injective embedded surface $S$ with genus 3 and $S\cdot S = 1$. 
\end{teo}

Since $S\cdot S$ is odd, the intersection form of $M$ is odd. The manifold $M$ is constructed by carefully assembling some copies of the right-angled 120-cell, along a pattern that was inspired to us by the minimum trisection of $\matCP^2$. The surface $S$ is contained in the 2-skeleton of $M$, which consists of many right-angled pentagons. Of course the surface $S$ is not totally geodesic: it is pleated along its edges and vertices, and its self-intersection $S\cdot S$ is calculated as the sum of the contributions of some rational weights assigned to its vertices via a beautiful formula of Gromov -- Lawson -- Thurston \cite{GLT}. Two vertices contribute each with $\frac 12$ while all others contribute with zero. So the sum is 1.

The 4-manifold $M$ that we construct is tessellated into right-angled 120-cells and is hence arithmetic of simplest type. By the Kolpakov -- Reid -- Slavich embedding theorem \cite{KRS}, the manifold $M$ totally geodesically embeds in a compact orientable arithmetic hyperbolic 5-manifold $M'$, that is hence also non-spin. By iterating this argument we find non-spin compact orientable arithmetic hyperbolic manifolds of simplest type in all dimensions $n\geq 4$.

\subsection*{Conclusions}
We briefly discuss here some consequences of Theorems \ref{main:teo} and \ref{odd:teo}. 

An even-dimensional compact hyperbolic manifold has non-zero Euler characteristic by the generalised Gauss -- Bonnet formula, so in particular it is never parallelisable. In odd dimensions,  Theorem \ref{main:teo} has the following consequence, which seems new, at least to our knowledge:

\begin{cor}
In every odd dimension $n\geq 5$ there are compact orientable hyperbolic manifolds that are not parallelisable.
\end{cor}

Restricting to dimension 4, the discussion above implies the following.

\begin{cor}
There is a compact orientable arithmetic hyperbolic 4-manifold $M$ such that $H_2(M,\matZ)$ is not generated by immersed totally geodesic surfaces.
\end{cor}

Note that the cohomology groups of small degree $k<\frac n3$ in compact arithmetic congruence hyperbolic $n$-manifolds are always generated by totally geodesic submanifolds \cite{BMM}. The pair $k=2, n=4$ is of course outside of this range.

In the manifold $M$ of Theorem \ref{odd:teo} the surface $S$ that we have found is $\pi_1$-injective. It is now reasonable to ask the following.

\begin{quest}
Let $M$ be a 
compact hyperbolic 4-manifold. Do $\pi_1$-injective orientable surfaces generate $H_2(M, \matZ)$?
\end{quest}

In our example, the fundamental group $\pi_1(S) < \pi_1(M)$ determines a covering $\tilde M \to M$ with $\tilde M = \matH^4/_{\pi_1(S)}$. We prove that $\tilde M$ is geometrically finite and diffeomorphic to a rank-two bundle over $S$ with Euler number 1. The existence of complete hyperbolic structures on non-trivial bundles over surfaces was first discovered by Gromov -- Lawson -- Thurston \cite{GLT} in 1988. The following consequence seems also new.

\begin{cor} \label{cor:bundles}
There are non-trivial bundles over surfaces that cover some compact hyperbolic 4-manifolds.
\end{cor}

\subsection*{Structure of the paper}
The construction of the non-spin compact hyperbolic 4-manifold $M$ is described in Section \ref{construction:section}. The proofs of two more technical lemmas are deferred to Section \ref{lemmas:section}. In Section \ref{pi:section} we show that $S$ is $\pi_1$-injective and that $\matH^4/_{\pi_1(S)}$ is geometrically finite and diffeomorphic to a plane bundle over $S$ with Euler number 1. Finally, in Section \ref{final:section} we complete the proof of Theorem \ref{main:teo} by passing from dimension 4 to any $n\geq 4$.
The reader interested only in the proof of Theorem \ref{main:teo} may skip Section \ref{pi:section}.

\subsection*{Acknowledgements}
The first author would like to thank Alan Reid, Daniel Ruberman, and Steven Tschantz for discussions on the topic. The second author was supported by the Swiss National Science Foundation (project no.~PP00P2-170560). He also thanks the Mathematics Department of the University of Pisa for the hospitality while this work was done.

\section{The construction} \label{construction:section}

Our goal is to construct a compact orientable hyperbolic 4-manifold $M$ that contains a surface $S$ with odd self-intersection. We plan to do this using right-angled polytopes, and in particular the right-angled 120-cell, that has already been employed to construct various hyperbolic manifolds and orbifolds (see for instance \cite{BM, KPV, KMT, MZ, Ma2, PP}). The right-angled 120-cell is of course a combinatorially complicated object, but thanks to its symmetries it may be used to construct hyperbolic manifolds effectively. 

We prove here Theorem \ref{odd:teo}, except for a couple of lemmas and the $\pi_1$-injectivity that are deferred respectively to Sections \ref{lemmas:section} and \ref{pi:section}.

\subsection{Pleated surfaces in right-angled 120-cell tessellations}
Our plan is to construct an oriented hyperbolic 4-manifold $M$ tessellated via a certain number of right-angled 120-cells, that contains $S$ in its 2-skeleton. Recall that the 120-cell has 120 facets that are right-angled dodecahedra. Therefore the 2-skeleton of $M$ is made of many right-angled pentagons. We want to construct $S$ as the union of some of these right-angled pentagons. 

As we said in the introduction, our desired surface $S\subset M$ cannot be totally geodesic, so it will be pleated along some of its edges and vertices, that are also edges and vertices of the tessellation of $M$. The surface $S$ will not be smooth, but it will be locally flat and easily smoothable. 

At every vertex $v$ of the tessellation of $M$, there are 16 (counted with multiplicity) 120-cells, whose link at $v$ form the standard triangulation of $S^3$ into 16 right-angled tetrahedra shown in Figure \ref{sfera:fig}. If $v$ is contained in $S$, its link $L_v$ is a closed unknotted curve contained in the 1-skeleton of this triangulation of $S^3$. Some cases are shown in Figure \ref{sfera_esempi:fig}. 

\begin{figure}
 \begin{center}
  \includegraphics[width = 6 cm]{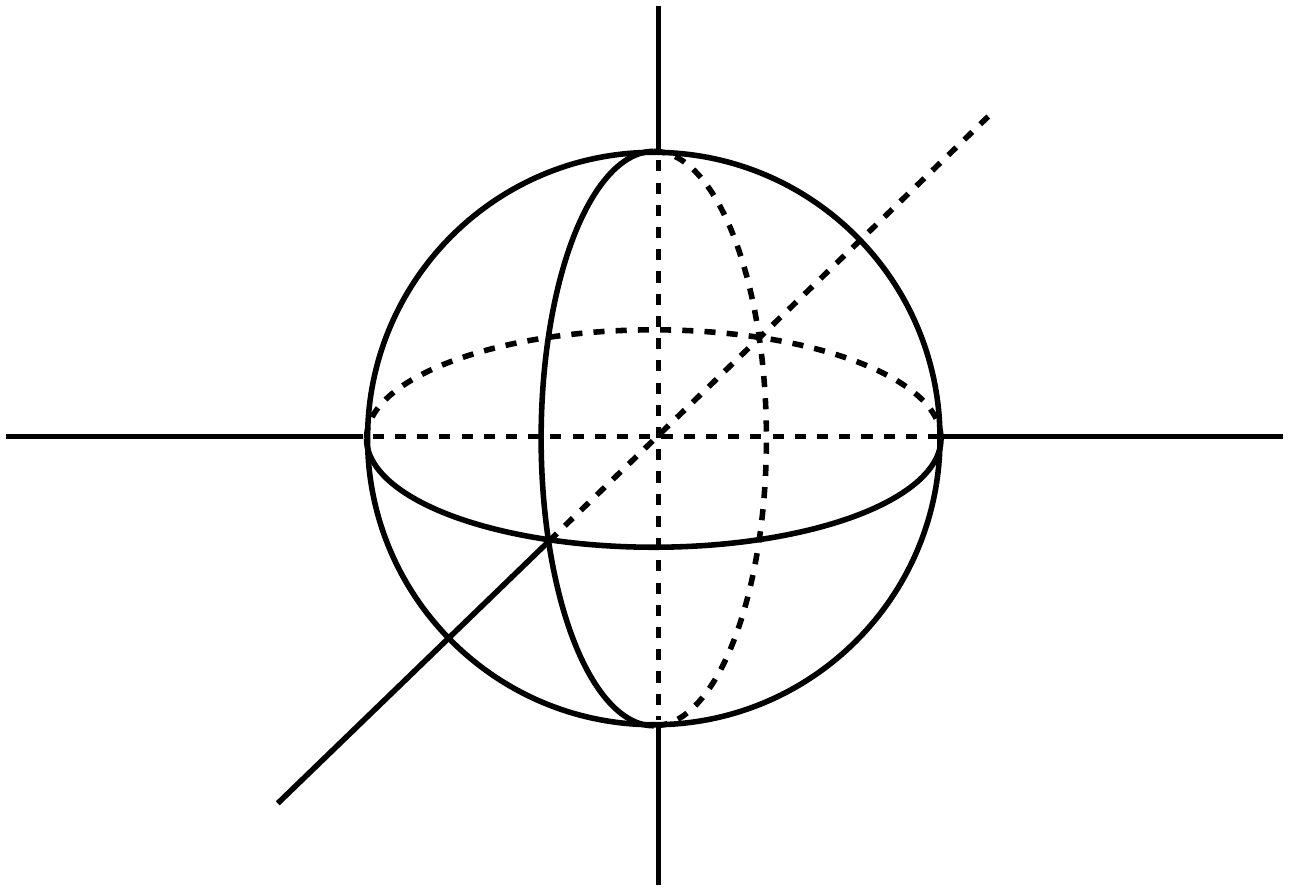}
 \end{center}
 \caption{The (stereographic projection of) the tessellation of $S^3$ into 16 right-angled tetrahedra.}
 \label{sfera:fig}
\end{figure}

\begin{figure}
 \begin{center}
  \includegraphics[width = 12.5 cm]{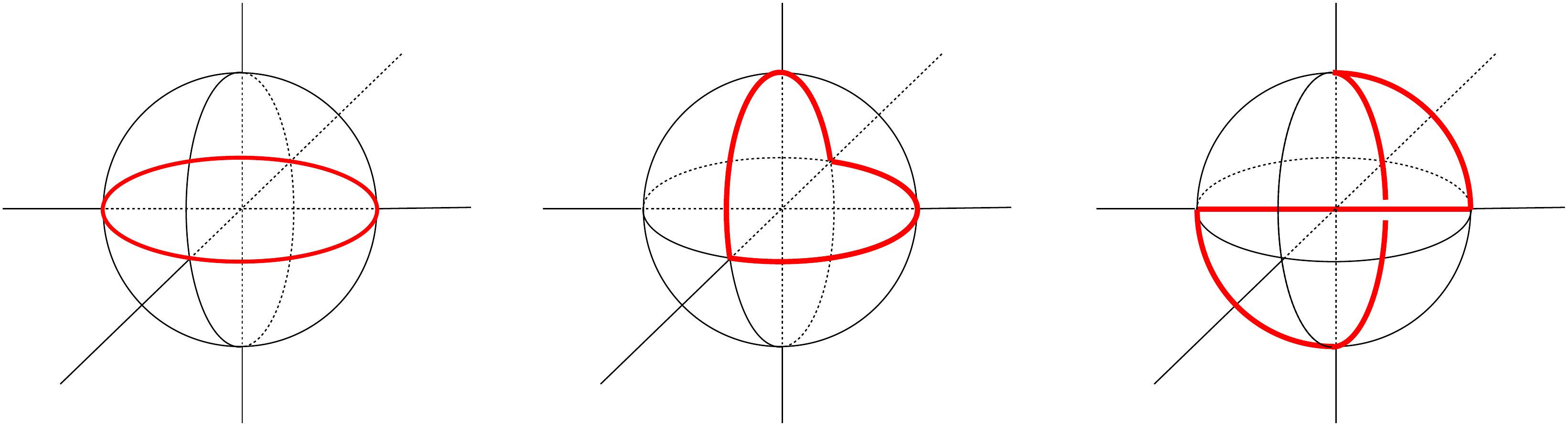}
 \end{center}
 \caption{The link $L_v$ of a vertex $v$ in the surface $S\subset M$ is a closed curve contained in the 1-skeleton of the tessellation of $S^3$ into 16 right-angled tetrahedra. Three examples (drawn in red) are shown here.}
 \label{sfera_esempi:fig}
\end{figure}

If $S$ is smooth at $v$, the link $L_v\subset S^3$ is a closed geodesic in $S^3$ as in Figure \ref{sfera_esempi:fig}-(left). If $S$ is pleated only along a geodesic arc containing $v$, the link $L_v$ is the union of two geodesic arcs as in Figure \ref{sfera_esempi:fig}-(centre). In general $L_v\subset S^3$ is a closed curve that consists of some geodesic arcs making right angles at some vertices $w_1, \ldots, w_k$ as in Figure \ref{sfera_esempi:fig}-(right). Each vertex $w_i$ points from $v$ towards an edge of the tessellation of $M$ incident to $v$ contained in $S$ where $S$ is bent (that is, it is not smooth). So the surface $S$ is bent along $k$ edges incident to $v$.

To calculate the self-intersection $S\cdot S$ of $S$, we use the beautiful simple formula of Gromov -- Lawson -- Thurston \cite{GLT}:
$$S\cdot S = \sum_v w(L_v).$$
Here $v$ runs over all the vertices of $S$, and $w(L)$ is a rational number (we call it the \emph{weight} of $L$) assigned to any closed curve $L\subset S^3$ contained in the 1-skeleton of the triangulation in Figure \ref{sfera:fig}. 

The weight $w(L)$ may be easily determined algorithmically (see \cite[Page 39] {GLT}), it is invariant under orientation-preserving isometries of the triangulation of $S^3$, while it changes by a sign under orientation-reversing ones. The three curves shown in Figure \ref{sfera_esempi:fig} have weight
$$0, \quad 0, \quad \frac 12$$
respectively.

\subsection{The Y-shaped piece}
We would like to construct a pair $(M,S)$ where $M$ is a compact oriented hyperbolic 4-manifold tessellated by right-angled $120$-cells and $S\subset M$ is a compact orientable surface contained in the 2-skeleton of $M$. We require $S$ to have two vertices with a link as in Figure \ref{sfera_esempi:fig}-(right), each contributing with weight $\frac 12$, while all other vertices contribute with zero. This will give $S\cdot S = 1$, as required.

Since it is much easier to construct surfaces inside a 3-dimensional environment than in a 4-dimensional one, we will build $S$ inside some reasonable 3-dimensional object $N$ contained in $M$. We will in fact construct a triple $S\subset N \subset M$.

What kind of reasonable 3-dimensional object $N$ can work for us? A first na\"\i ve request could be to take $N$ as an orientable 3-dimensional submanifold in $M$. This request however would be too restrictive: if $S$ is contained in an orientable 3-dimensional submanifold $N$ of $M$, its normal bundle in $M$ is trivial and hence $S\cdot S$ is zero. 

As a second try, we require $N\subset M$ to be a \emph{Y-shaped piece}, that is a kind of generalised trisection, where three orientable 3-manifolds $N_0, N_1, N_2$ are glued along a common boundary surface $\Sigma$ as in Figure \ref{Tstratum:fig}. Here is the precise definition:

\begin{defn}
Let $M$ be a smooth 4-manifold. 
A \emph{Y-shaped piece} is a subset $N\subset M$ which decomposes into three portions
$$N = N_0 \cup N_1 \cup N_2$$
as in Figure \ref{Tstratum:fig}, where:
\begin{enumerate}
\item each $N_i$ is a smooth 3-dimensional orientable submanifold with boundary;
\item the intersection $\Sigma = N_0 \cap N_1 = N_1 \cap N_2 = N_2 \cap N_0$ is a boundary component of each $N_i$.
\end{enumerate}
\end{defn}

\begin{figure}
 \begin{center}
 \labellist
\small\hair 2pt
\pinlabel $N_1$ at 10 52
\pinlabel $N_2$ at 72 52
\pinlabel $N_0$ at 48 10
\pinlabel $\Sigma$ at 41 54
\endlabellist
  \includegraphics[width = 2.5 cm]{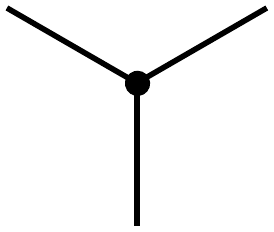}
 \end{center}
 \caption{A Y-shaped piece consists of three orientable 3-dimensional smooth submanifolds $N_0, N_1, N_2$ with boundary, intersecting in a common boundary surface $\Sigma$.}
 \label{Tstratum:fig}
\end{figure}
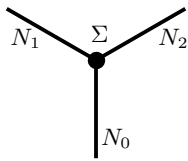

Note that each manifold $N_i$ is allowed to have some additional boundary components other than $\Sigma$. A $Y$-shaped piece is in some sense the simplest kind of 3-dimensional object that is not a manifold. We call $\Sigma$ the \emph{central} surface of the $Y$-shaped piece. 

Let $M$ be an oriented 4-manifold that contains a Y-shaped piece $N=N_0\cup N_1 \cup N_2$ with central surface $\Sigma$. We now describe a simple homological condition that guarantees that $M$ contains a surface $S$ with $S\cdot S = \pm 1$. Let $\Theta \subset \Sigma$ be a \emph{$\theta$-graph},
that is a $\theta$-shaped 1-complex as in Figure \ref{theta:fig} whose regular neighborhood in $\Sigma$ is a punctured torus. The graph $\Theta$ contains three oriented simple closed curves $\gamma_0, \gamma_1, \gamma_2$ with $[\gamma_0] + [\gamma_1] + [\gamma_2]=0$ in homology. 

\begin{figure}
 \begin{center}
 \labellist
\small\hair 2pt
\pinlabel $\Sigma$ at 5 5
\pinlabel $\Theta$ at 232 50
\endlabellist
  \includegraphics[width = 7 cm]{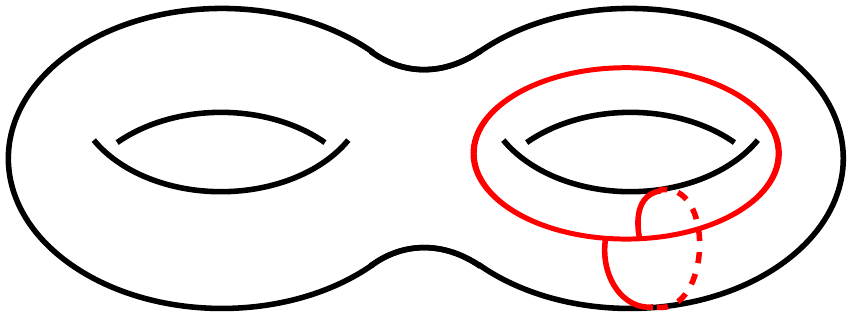}
 \end{center}
 \caption{A $\theta$-graph $\Theta$ in the central surface $\Sigma$.}
 \label{theta:fig}
\end{figure}

\begin{figure}
 \begin{center}
 \labellist
\small\hair 2pt
\pinlabel $S_2$ at 2 2
\pinlabel $S_0$ at 260 2
\pinlabel $S_1$ at 190 200
\pinlabel $\Theta$ at 180 118
\endlabellist
  \includegraphics[width = 6 cm]{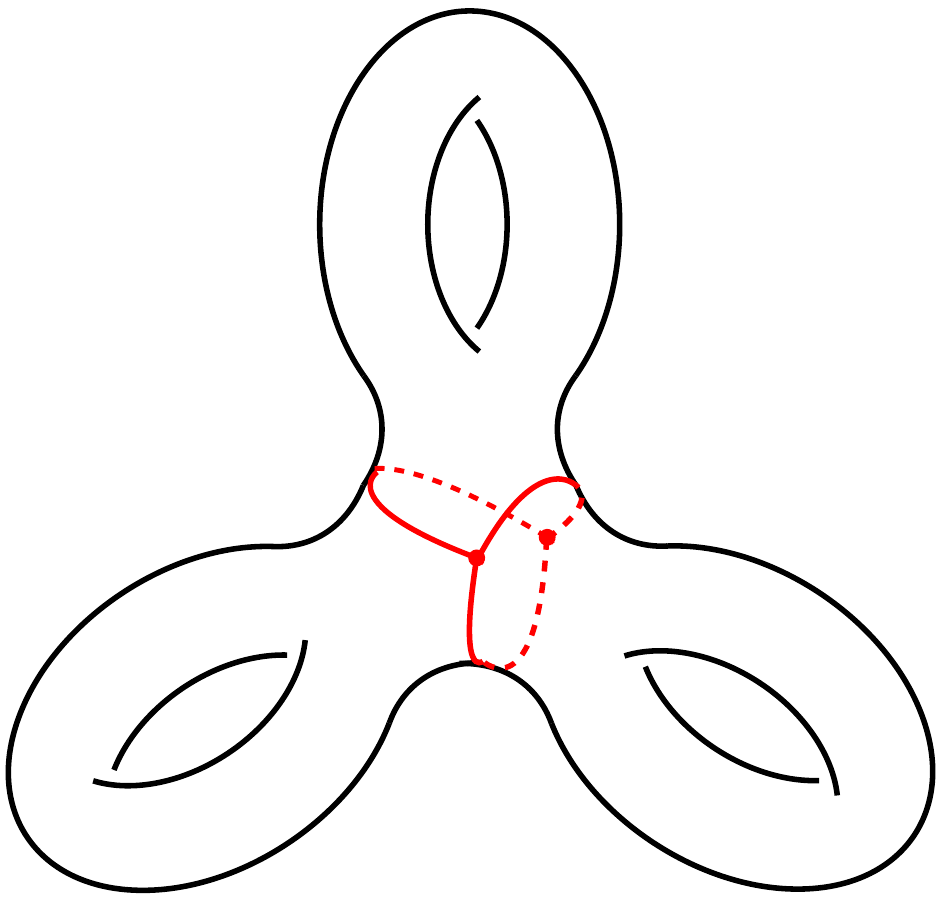}
 \end{center}
 \caption{The graph $\Theta$ in $S=S_0\cup S_1\cup S_2$.}
 \label{S:fig}
\end{figure}

\begin{figure}
 \begin{center}
 \labellist
\small\hair 2pt
\pinlabel $N_1$ at 10 52
\pinlabel $N_2$ at 72 52
\pinlabel $N_0$ at 48 10
\pinlabel $\Sigma$ at 30 38
\pinlabel $N_1'$ at 30 78
\pinlabel $N_2'$ at 89 75
\pinlabel $N_0'$ at 58 30
\pinlabel $\Sigma'$ at 55 70
\endlabellist
  \includegraphics[width = 3 cm]{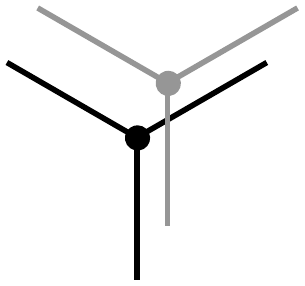}
 \end{center}
 \caption{A Y-shaped piece $N$ together with a copy $N'$ slightly pushed in a random direction. Here $N\cap N' = N_2\cap N_0'$ is a surface parallel to both $\Sigma$ and $\Sigma'$.}
 \label{Tstratum_perturbed:fig}
\end{figure}

\begin{prop} \label{SS1:prop}
Suppose that each $\gamma_i$ is the boundary of a properly embedded compact oriented surface $S_i \subset N_i$, for all $i=0,1,2$. Then $S = S_0\cup S_1 \cup S_2$ is a closed oriented surface in $M$ with $S\cdot S = \pm 1$.
\end{prop}
\begin{proof}
We first note that indeed $S$ is a closed oriented surface, pleated along $\Theta$ and easily smoothable, see Figure \ref{S:fig}. To calculate $S\cdot S$, we push $N$ slightly in some random direction as in Figure \ref{Tstratum_perturbed:fig}, to produce a new $Y$-shaped piece $N'$. Note that $N\cap N' = N_2 \cap N_0'$ is a surface parallel to both $\Sigma$ and $\Sigma'$. After a slight perturbation the isotopic copy $S'\subset N'$ of $S\subset N$ intersects $S$ transversely in a single point, that corresponds to the transverse intersection of the perturbed curves $\gamma_0$ and $\gamma_2$ in $\Sigma$.
\end{proof}

The hypothesis of Proposition \ref{SS1:prop} is in fact just a homological condition on each $\gamma_i$: we require $\gamma_i$ to be zero in $H_1(N_i, \matZ)$ for all $i$. This condition guarantees the existence of the surfaces $S_i$. In our construction, every $S_i$ will be a one-holed torus and hence $S$ will have genus 3 as in Figure \ref{S:fig}.

\begin{example}
The genus-one trisection $N$ of $\matCP^2$ satisfies these hypothesis and was in fact our main inspiration. The trisection $N$ is a $Y$-shaped piece made of 3 solid tori $N_0, N_1, N_2$ in $\matCP^2$ with a common boundary torus $\Sigma$. We may choose meridian discs $S_i \subset N_i$ whose boundary curves $\gamma_i = \partial S_i$ are contained in a $\theta$-graph $\Theta \subset \Sigma$. The three meridians glue to form a sphere $S=S_0 \cup S_1 \cup S_2$ with $S\cdot S = 1$. The sphere $S$ is isotopic to a line in $\matCP^2$. 
\end{example}

Our strategy to construct the hyperbolic manifold $M$ is now the following: we first build an abstract geometric Y-shaped piece $N$ made of right-angled dodecahedra, and then we enlarge $N$ to a compact hyperbolic 4-manifold $M$ by adding right-angled 120-cells. 

\subsection{Proof of Theorem \ref{odd:teo}}
We now prove Theorem \ref{odd:teo}. The proof of two more technical lemmas and of the $\pi_1$-injectivity will be deferred to Sections \ref{lemmas:section} and \ref{pi:section}.

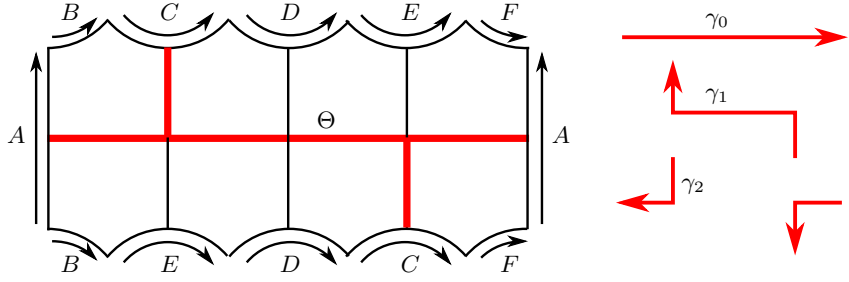
\begin{figure}
 \begin{center}
 \labellist
\small\hair 2pt
\pinlabel $A$ at -3 45
\pinlabel $A$ at 178 45
\pinlabel $B$ at 15 86
\pinlabel $B$ at 15 2
\pinlabel $C$ at 48 86
\pinlabel $E$ at 48 2
\pinlabel $D$ at 88 86
\pinlabel $D$ at 88 2
\pinlabel $E$ at 128 86
\pinlabel $C$ at 128 2
\pinlabel $F$ at 161 86
\pinlabel $F$ at 161 2
\pinlabel $\Theta$ at 100 50
\pinlabel $\gamma_0$ at 230 83
\pinlabel $\gamma_1$ at 230 58
\pinlabel $\gamma_2$ at 222 28
\endlabellist
  \includegraphics[width = 11 cm]{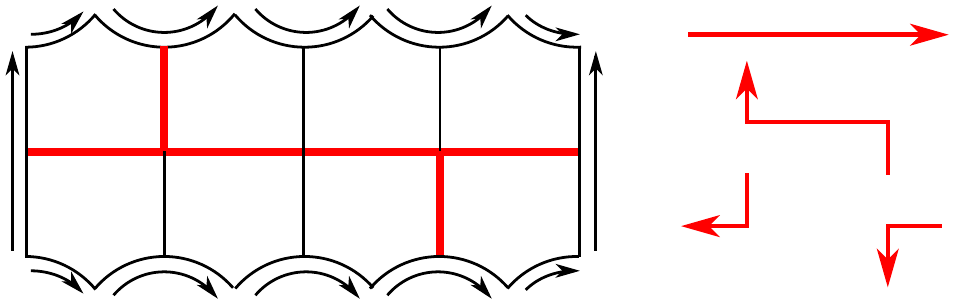}
 \end{center}
 \caption{The surface $\Sigma$ is the hyperbolic genus-two surface tessellated into 8 right-angled pentagons as shown here. The edges labeled with the same letters should be paired isometrically according to the arrows (left). The $\theta$-graph $\Theta$ contains three oriented simple closed curves $\gamma_0$, $\gamma_1$, and $\gamma_2$ (right).}
 \label{sigma:fig}
\end{figure}

Let $\Sigma$ be the genus-two oriented hyperbolic surface tessellated into 8 right-angled pentagons shown in Figure \ref{sigma:fig}-(left). The surface $\Sigma$ contains a $\theta$-graph $\Theta$, drawn in red in the figure. Let $\gamma_0, \gamma_1, \gamma_2$ be the three
oriented simple closed curves contained in $\Theta$ as shown in Figure \ref{sigma:fig}-(right). We have $[\gamma_0]+[\gamma_1]+[\gamma_2]=0$ in homology. 

The following purely 3-dimensional lemma says that $\Sigma$ is (part of) the geodesic boundary of a hyperbolic 3-manifold $N_i$ containing a one-holed torus $S_i$ with $\partial S_i = \gamma_i$, for all $i$. The manifold $N_i$ is nicely tessellated into dodecahedra.

\begin{lemma} \label{Ni:lemma}
There are three compact oriented hyperbolic 3-manifolds $N_0, N_1, N_2$ with geodesic boundary, tessellated into right-angled dodecahedra, such that:
\begin{enumerate}
\item one boundary component of $N_i$ is isometrically identified to $\Sigma$, with an isometry that preserves the tessellations into pentagons, for all $i$;
\item the boundary component $\Sigma\subset \partial N_i$ is \emph{nicely collared}, that is the 8 dodecahedra in $N_i$ incident to the 8 pentagons of $\Sigma$ are all distinct, for all $i$;
\item there is a properly embedded oriented one-holed torus $S_i \subset N_i$ with boundary $\partial S_i = \gamma_i$, for all $i$.
\end{enumerate}
\end{lemma}

We now construct an abstract geometric Y-shaped $N$ by glueing $N_0$, $N_1$, and $N_2$ to $\Sigma$. Starting from $N$, we may thicken it and then close it to a hyperbolic 4-manifold $M$ using 120-cells. This is how we prove the following lemma.

\begin{lemma} \label{M:lemma}
There is a compact orientable hyperbolic 4-manifold $M$ containing $N$ as a Y-shaped piece. 
\end{lemma}

The proof of the main part of Theorem \ref{odd:teo} is now complete: by construction the manifold $M$ contains the Y-shaped piece $N$, which in turn contains three surfaces $S_0, S_1, S_2$ whose union $S$ has genus 3 and $S\cdot S = \pm 1$ by Proposition \ref{SS1:prop}. We get $S\cdot S = 1$ by choosing the appropriate orientation for $M$. By construction $M$ is arithmetic, as explained at the beginning of Section \ref{final:section}.

The two lemmas are proved in the next section. In Section \ref{pi:section} we show that $S$ is also $\pi_1$-injective, 
and this will conclude the proof of Theorem \ref{odd:teo}.

The surface $S$ is shown in Figure \ref{S_tessellated:fig}. We provide another proof of the equality $S\cdot S = \pm 1$ in Section \ref{cutting:subsection} via the Gromov -- Lawson -- Thurston formula.

\section{Proofs of the lemmas} \label{lemmas:section}
We prove here Lemmas \ref{Ni:lemma} and \ref{M:lemma}. Their proofs are similar: in both cases we construct some hyperbolic manifolds of dimension 3 or 4 by attaching right-angled dodecahedra or 120-cells to some existing object, that is the surface $\Sigma$ or the Y-shaped piece $N$. We introduce a general definition, taken from \cite{Ma2}.

\subsection{Hyperbolic manifolds with corners}
We recall from \cite{Ma2} the notion of \emph{hyperbolic manifold with (right-angled) corners}, that generalises both hyperbolic manifolds with geodesic boundary and right-angled polytopes. 

We use the Klein model $D^n$ for hyperbolic space and define $P\subset D^n$ as the intersection of $D^n$ with the positive sector $x_1, \ldots, x_n \geq 0$. 
A \emph{hyperbolic manifold with (right-angled) corners} is a topological $n$-manifold $M$ with an atlas in $P$ and transition maps that are restrictions of isometries. The boundary $\partial M$ is naturally stratified into connected closed $k$-dimensional strata called \emph{faces}, that we call vertices, edges, and facets, if $k=0,1$, and $n-1$ respectively.
Every face is abstractly itself a hyperbolic $k$-manifold with corners; note that a face may not be embedded, because it may be incident multiple times to the same lower-dimensional face. 

A manifold with corners can also be interpreted as an orbifold with mirrors, but we do not really need the more general orbifold language here: everything will be elementary.

As we said, hyperbolic manifolds with geodesic boundary and right-angled polytopes are particular kinds of hyperbolic manifolds with corners. One crucial property of this class of objects is the following: if we glue two hyperbolic manifolds with corners along two isometric embedded facets, the result is naturally a new hyperbolic manifold with corners.
More generally, if we glue two disjoint embedded isometric facets of a (possibly disconnected) hyperbolic manifold with corners, we get a new hyperbolic manifold with corners.

A nice operation that we can do with a manifold $M$ with corners is \emph{colouring and mirroring}. 
Suppose that we can colour some of the embedded facets of $M$, in such a way that adjacent coloured facets always get different colours. Then we can mirror $M$ iteratively along the facets having the same colour, and get at the end a bigger manifold with corners $M'$ containing $M$.
If we have coloured all the facets of $M$, the resulting $M'$ is without boundary. If $M$ is oriented, also $M'$ is. 
See \cite[Proposition 6]{Ma2} for more details. Using the orbifold language, we have constructed a finite orbifold covering $M' \to M$. The manifold $M'$ is tessellated into $2^k$ copies of $M$, where $k$ is the number of colours in our palette.

Note that if $M$ has $k$ facets, and these are all embedded, we can colour them with $k$ different colours: this will produce a compact hyperbolic manifold $M'$ without boundary tessellated into $2^k$ copies of $M$. 

\subsection{Proof of Lemma \ref{Ni:lemma}} \label{Ni:sec}
Up to symmetry, it suffices to consider the curves $\gamma_0$ and $\gamma_1$ shown in Figure \ref{sigma_casi:fig}, since $\gamma_2$ is isometric to $\gamma_1$. In both cases we start by attaching 8 right-angled dodecahedra above $\Sigma$, one above each pentagon. The result is a hyperbolic 3-manifold with corners, with two boundary components: its bottom is the totally geodesic $\Sigma$, while its top is isotopic to $\Sigma$ and pleated at right-angles along the pattern shown in Figure \ref{sigma_casi2:fig}. The top contains 10 octagons and 8 pentagons. 

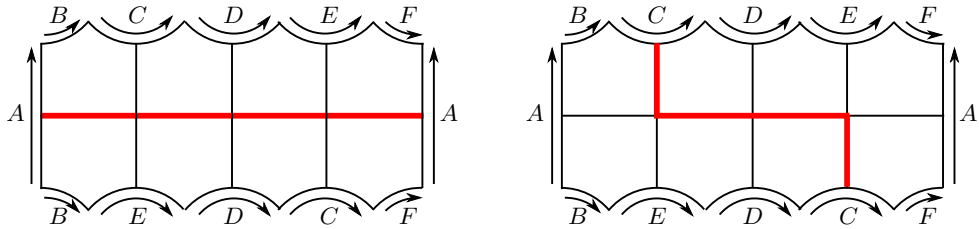
\begin{figure}
 \begin{center}
 \labellist
\small\hair 2pt
\pinlabel $A$ at -3 45
\pinlabel $A$ at 178 45
\pinlabel $B$ at 15 86
\pinlabel $B$ at 15 2
\pinlabel $C$ at 48 86
\pinlabel $E$ at 48 2
\pinlabel $D$ at 88 86
\pinlabel $D$ at 88 2
\pinlabel $E$ at 128 86
\pinlabel $C$ at 128 2
\pinlabel $F$ at 161 86
\pinlabel $F$ at 161 2
\pinlabel $A$ at 214 45
\pinlabel $A$ at 395 45
\pinlabel $B$ at 232 86
\pinlabel $B$ at 232 2
\pinlabel $C$ at 265 86
\pinlabel $E$ at 265 2
\pinlabel $D$ at 305 86
\pinlabel $D$ at 305 2
\pinlabel $E$ at 345 86
\pinlabel $C$ at 345 2
\pinlabel $F$ at 378 86
\pinlabel $F$ at 378 2
\endlabellist
  \includegraphics[width = 12.5 cm]{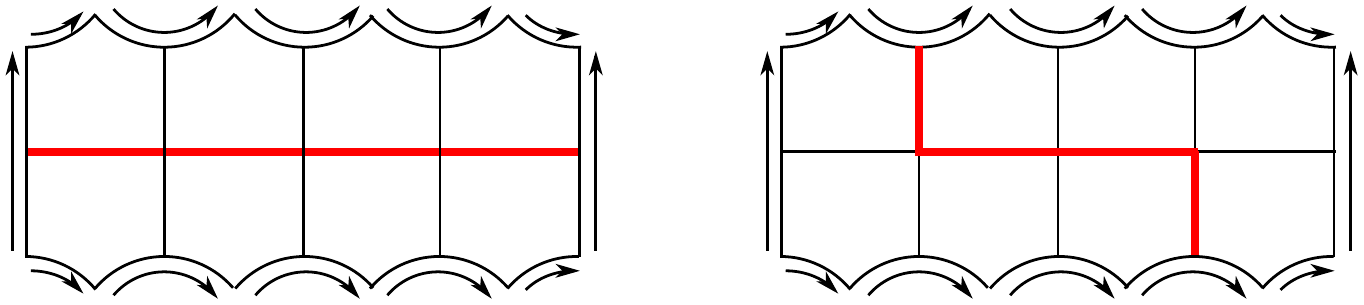}
 \end{center}
 \caption{The curves $\gamma_0$ and $\gamma_1$ in $\Sigma$.}
 \label{sigma_casi:fig}
\end{figure}

\begin{figure}
 \begin{center}
 \labellist
\small\hair 2pt
\pinlabel $A$ at -3 45
\pinlabel $A$ at 178 45
\pinlabel $B$ at 15 86
\pinlabel $B$ at 15 2
\pinlabel $C$ at 48 86
\pinlabel $E$ at 48 2
\pinlabel $D$ at 88 86
\pinlabel $D$ at 88 2
\pinlabel $E$ at 128 86
\pinlabel $C$ at 128 2
\pinlabel $F$ at 161 86
\pinlabel $F$ at 161 2
\pinlabel $A$ at 214 45
\pinlabel $A$ at 395 45
\pinlabel $B$ at 232 86
\pinlabel $B$ at 232 2
\pinlabel $C$ at 265 86
\pinlabel $E$ at 265 2
\pinlabel $D$ at 305 86
\pinlabel $D$ at 305 2
\pinlabel $E$ at 345 86
\pinlabel $C$ at 345 2
\pinlabel $F$ at 378 86
\pinlabel $F$ at 378 2
\endlabellist
  \includegraphics[width = 12.5 cm]{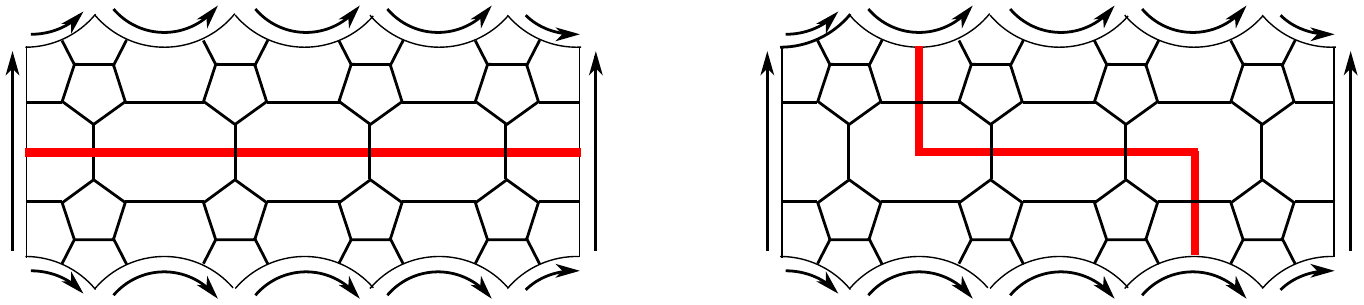}
 \end{center}
 \caption{The result of attaching 8 dodecahedra to $\Sigma$ is to produce a hyperbolic 3-manifold with corners, whose bottom is $\Sigma$, and whose top is bent at right-angles along the pattern shown here.}
 \label{sigma_casi2:fig}
\end{figure}

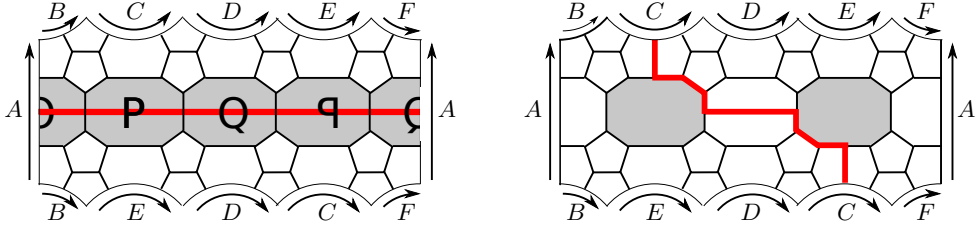
\begin{figure}
 \begin{center}
 \labellist
\small\hair 2pt
\pinlabel $A$ at -3 45
\pinlabel $A$ at 178 45
\pinlabel $B$ at 15 86
\pinlabel $B$ at 15 2
\pinlabel $C$ at 48 86
\pinlabel $E$ at 48 2
\pinlabel $D$ at 88 86
\pinlabel $D$ at 88 2
\pinlabel $E$ at 128 86
\pinlabel $C$ at 128 2
\pinlabel $F$ at 161 86
\pinlabel $F$ at 161 2
\pinlabel $A$ at 214 45
\pinlabel $A$ at 395 45
\pinlabel $B$ at 232 86
\pinlabel $B$ at 232 2
\pinlabel $C$ at 265 86
\pinlabel $E$ at 265 2
\pinlabel $D$ at 305 86
\pinlabel $D$ at 305 2
\pinlabel $E$ at 345 86
\pinlabel $C$ at 345 2
\pinlabel $F$ at 378 86
\pinlabel $F$ at 378 2
\endlabellist
  \includegraphics[width = 12.5 cm]{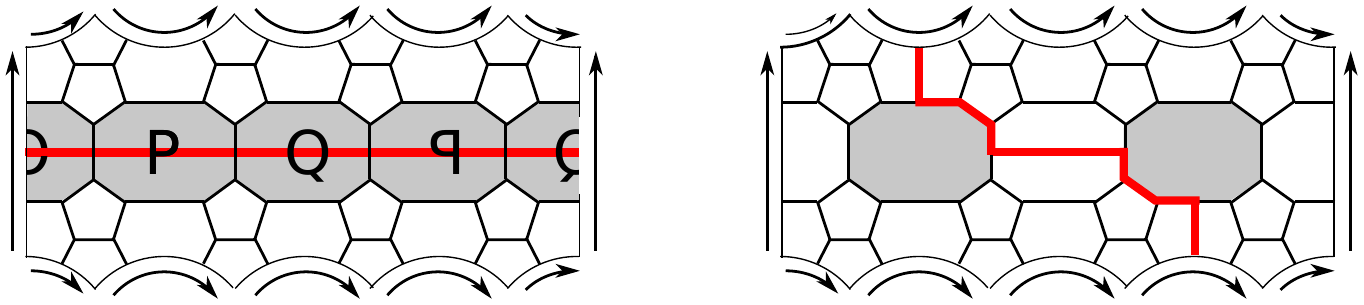}
 \end{center}
 \caption{We pair the 4 grey octagons isometrically, as indicated by the letters P and Q. The result is a new hyperbolic manifold with corners (left). We isotope the curve (right).}
 \label{sigma_casi3:fig}
\end{figure}

In the $\gamma_0$ case, we identify isometrically two pairs of top octagons as shown in Figure \ref{sigma_casi3:fig}-(left). 
We end up with an oriented manifold $\bar N_0$ with corners that contains a totally geodesic punctured torus with boundary on $\gamma_0$.
We can easily check that every face in $\bar N_0$ is embedded.
We can colour arbitrarily all its faces except $\Sigma$ (for instance, by assigning different colours to distinct facets), and then double $\bar N_0$ iteratively along its coloured facets. At the end we get an oriented manifold $N_0 \supset \bar N_0$ with totally geodesic boundary, that consists of the original $\Sigma$ and of many other copies of $\Sigma$ that will not be important for us. 

In the $\gamma_1$ case we would like to follow the same strategy but we encounter some additional technicalities because $\gamma_1$ is pleated. We cannot do a similar pairing, for the following reason: in order to build an orientable surface inside an orientable $3$-manifold, we would need the pairing maps between facets to be orientation reversing both on the facets and on the pleated red curve isotopic to $\gamma_1$ shown in Figure \ref{sigma_casi2:fig}-(right). There is no such isometry between the octagons which contain the pleating points of $\gamma_1$. In order to overcome this problem, we isotope $\gamma_1$ as shown in Figure \ref{sigma_casi3:fig}-(right). Then, we attach 4 dodecahedra above each of the two grey octagons shown in Figure \ref{sigma_casi3:fig}-(right). Let us call $\bar N'_1$ the resulting hyperbolic manifold with corners. By an accurate analysis we discover that the top of $\bar N'_1$ is as in Figure \ref{sigma_casi4:fig}.

We would like to pair the 4 grey facets as shown in Figure \ref{sigma_casi4:fig} and get as above a manifold with corners $\bar N_1$ containing a punctured torus with boundary on $\gamma_1$.
This can be done, but unfortunately a new difficulty emerges: the resulting manifold with corners $\bar N_1$ has a non-embedded facet, because all the facets in $\bar N_1$ labeled with 1 or 2 in Figure \ref{sigma_casi4:fig} glue up to a single non-embedded facet in $\bar N_1$. Non-embedded facets cannot be coloured, so we cannot conclude as we did with $N_0$.

To solve this problem we make a more complicated construction. We colour the problematic facets of $\bar N_1'$ with two colours (1 and 2) as indicated in Figure \ref{sigma_casi4:fig}. Specifically: we assign the colour 1 to two pentagons and two octagons, and the colour 2 to two pentagons and one octagon. We then mirror $\bar N'_1$ twice according to the colouring. Let us call $\bar N''_1$ the resulting manifold with corners, tessellated by four copies of $\bar N'_1$.

Every grey facet of $\bar N_1'$ labeled by either P or F is contained in a bigger facet of $\bar N_1''$. There is a unique way to pair isometrically these bigger facets of $\bar N_1''$ so that the original grey facets of $\bar N_1'$ match as we desired (this holds because the colouring was chosen to be compatible with that). If we pair them we obtain a new manifold with corners $\bar N_1$. It is not difficult to check that every facet of $\bar N_1$ is embedded. So now we conclude as we did for $\bar N_0$, that is we build $N_1$ from $\bar N_1$ by colouring everything except $\Sigma$ and mirroring.

\begin{figure}
 \begin{center}
 \labellist
\small\hair 2pt
\pinlabel $A$ at -3 45
\pinlabel $A$ at 178 45
\pinlabel $B$ at 15 86
\pinlabel $B$ at 15 2
\pinlabel $C$ at 48 86
\pinlabel $E$ at 48 2
\pinlabel $D$ at 88 86
\pinlabel $D$ at 88 2
\pinlabel $E$ at 128 86
\pinlabel $C$ at 128 2
\pinlabel $F$ at 161 86
\pinlabel $F$ at 161 2
\endlabellist
  \includegraphics[width = 11 cm]{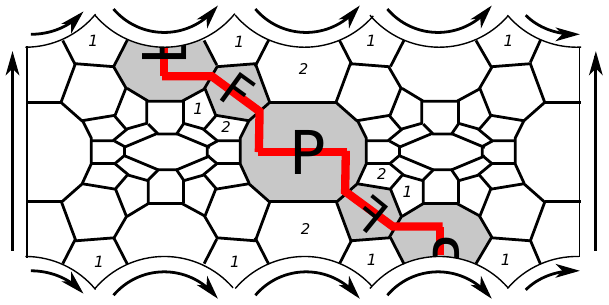}
 \end{center}
 \caption{The top of $\bar N'_1$. The pairing of the 4 grey facets as indicated by the letters P and F would produce a big non-embedded facet, because all the facets labeled with 1 or 2 glue up in the process. To avoid this, we first mirror $\bar N'_1$ twice according to the chosen $\{1,2\}$-colouring of these facets, and get $\bar N''_1$. After that, we pair the 4 new facets of $\bar N''_1$ containing the 4 grey facets to get $\bar N_1$.}
 \label{sigma_casi4:fig}
\end{figure}

\subsection{Proof of Lemma \ref{M:lemma}} \label{M:sec}
Let $N = N_0 \cup N_1 \cup N_2$ be an abstract Y-shaped piece constructed by attaching the three 3-manifolds with geodesic boundary $N_0, N_1, N_2$ to the surface $\Sigma$, via an isometry that preserves the tessellations into pentagons. 

We now construct a hyperbolic 4-manifold with corners $\bar M$ by attaching 120-cells to $N$ in a similar fashion as in \cite{Ma2}.

\begin{figure}
 \begin{center}
 \labellist
\small\hair 2pt
\pinlabel $\Sigma$ at 97 170
\pinlabel $N_1$ at 60 190
\pinlabel $N_2$ at 160 190
\pinlabel $N_0$ at 97 88
\pinlabel $\Sigma$ at 345 170
\pinlabel $\Sigma$ at 345 115
\pinlabel $N_1$ at 308 190
\pinlabel $N_2$ at 408 190
\pinlabel $N_0$ at 345 28
\endlabellist
  \includegraphics[width = 11 cm]{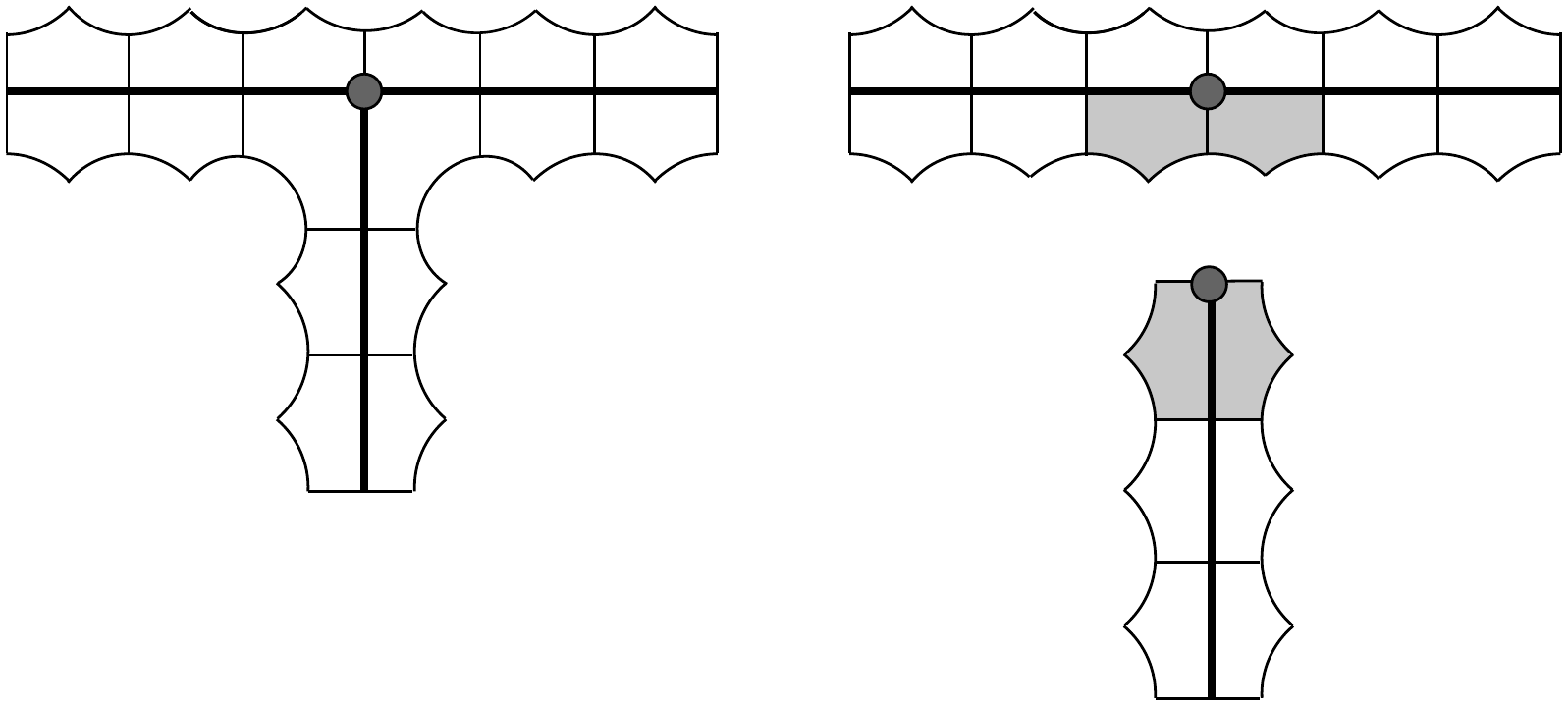}
 \end{center}
 \caption{We embed the Y-shaped piece $N$ in a hyperbolic 4-manifold with corners by forming an abstract regular neighbourhood of 120-cells. Here we draw the construction in dimension 2, with segments and pentagons instead of dodecahedra and 120-cells (left). This may be seen as a two-step procedure, where we first consider $N_1\cup N_2$ and $N_0$ separately and then we identify the grey 120-cells (right).}
 \label{Testendi:fig}
\end{figure}

We visualise $N$ geometrically as in Figure \ref{Testendi:fig}-(left): we first glue $N_1$ and $N_2$ along $\Sigma$, so that $N_1 \cup N_2$ is a hyperbolic 3-manifold with geodesic boundary containing $\Sigma$ in its interior; then we attach $N_0$ to $\Sigma$ making (in an abstract sense) an angle $\frac \pi 2$ with $N_1\cup N_2$.

Our aim is to construct an abstract ``regular neighbourhood'' of $N$ by attaching 120-cells to the dodecahedra as sketched in Figure \ref{Testendi:fig}-(left). The construction goes as in Figure \ref{Testendi:fig}-(right): we consider the hyperbolic 3-manifolds $N_1 \cup N_2$ and $N_0$ with geodesic boundary separately. These manifolds decompose into right-angled dodecahedra, so as in \cite{Ma2} we may attach two 120-cells to each dodecahedron (one ``above'' and the other ``below'') and get two hyperbolic 4-manifolds with corners that contain $N_1 \cup N_2$ and $N_0$, respectively. (We can do this unambiguously because every isometry of a dodecahedral facet extends uniquely to an isometry of the 120-cell.) 

Now we identify in pairs the 120-cells in $N_0$ incident to $\Sigma$ with the 120-cells in $N_1 \cup N_2$ that are incident to $\Sigma$ from below, as in Figure \ref{Testendi:fig}-(right). There is a natural unambiguous way to do this, as suggested by the figure. Note that since the manifolds $N_i$ are nicely collared all the 120-cells involved are indeed distinct. 

After this identification, we get a manifold with boundary $\bar M$, that may be interpreted as a regular neighbourhood of $N$, as suggested by Figure \ref{Testendi:fig}-(left). We make a crucial observation: the manifold $\bar M$ is still a hyperbolic 4-manifold with right angled corners. 

To see this, consider the tessellation of $\bar M$ into copies of the $120$-cell, and choose a pentagonal face $F$ lying in the surface $\Sigma$. Now, consider one of the $120$-cells  which contains $F$ and intersects the $3$-manifold $N_0$. In this $120$-cell $C$ there are two dodecahedra $D_1$ and $D_2$ which contain $F$. One of the two dodecahedra, say $D_1$, is contained in $N_0$, while $D_2$ is contained in either $N_1$ or $N_2$. All the other dodecahedra in $C$ are either incident to both $D_1$ and $D_2$ (there are five such dodecahedra), or they are incident to $D_1$ but not to $D_2$, or they are incident to $D_2$ and not to $D_1$, or are disjoint from both $D_1$ and $D_2$. 

Any dodecahedron $D'$ which intersects $D_1$ and not $D_2$ is not incident to any dodecahedron $D''$ which intersects $D_2$ and not $D_1$ (this can be checked with some patience by looking at the combinatorics of the 120-cell). This fact is crucial here: if this were not the case, there would be two $120$-cells $C_1$ and $C_2$ in $\bar M$, with $C_i$ adjacent to $C$ along $D_i$, $i=1,2$, with the property that the total interior angle along their common pentagonal intersection would be equal to the forbidden angle $\frac{3\pi}2$. Note that this bad configuration arises in flat geometry if we use hyper-cubes on cubes instead of 120-cells on dodecahedra. 

We have thus proved that in the boundary of $\bar M$ no pair of facets intersect with the forbidden interior angle $\frac{3\pi}2$. Therefore all interior angles in the boundary of $\bar M$ are in fact right angles and $\bar M$ is a genuine hyperbolic manifold with corners.
Finally, by colouring arbitrarily the facets of $\bar M$ and then mirroring we get a bigger compact orientable hyperbolic manifold $M$ without boundary containing $N$.

\section{The surface subgroup $\pi_1(S)$}  \label{pi:section}

In this section we prove the following:

\begin{prop}\label{pi_1(S):prop}
The surface $S$ is $\pi_1$-injective in $M$,
the group $\pi_1(S)<\mathrm{Isom} (\mathbb{H}^4)$ is geometrically finite, and $\tilde M=\mathbb{H}^4/_{\pi_1(S)}$ is diffeomorphic to the total space of the rank-two vector bundle over $S$ of Euler number one.
\end{prop}

In particular, the $\pi_1$-injectivity of $S$ will conclude the proof of Theorem \ref{odd:teo}, and the covering $\tilde M\to M$ will prove Corollary \ref{cor:bundles}.

The strategy to prove Proposition \ref{pi_1(S):prop} is to exhibit a convex fundamental domain $\mathcal{D}$ for the action of $\pi_1(S)$ on $\mathbb{H}^4$ induced by the inclusion $S\subset M$ which, \emph{a priori}, is not necessarily faithful. The fundamental domain $\mathcal D$ will be a right-angled \emph{convex 20-gon}, as defined by Kuiper \cite{K}: this is a polyhedron with 20 cyclically consecutive facets, each isometric to the complement in $\matH^3$ of two open half-spaces with disjoint closures in $\overline \matH^3$. The domain $\mathcal D$ is tessellated into infinitely many right-angled $120$-cells. We use Poincar\'e's Fundamental Polyhedron Theorem to prove that $\mathcal D$ is indeed a fundamental domain and that the action of $\pi_1(S)$ is faithful.

Since $\mathcal{D}$ is a finite-sided polytope, the manifold $\tilde M = \matH^4/_{\pi_1(S)}$ is geometrically finite. 
Moreover, $\mathcal{D}$ is homeomorphic to the product $D^2\times \mathbb{R}^2$, and the pairing maps preserve both the boundary of the disc $D^2\times\{(0,0)\}$ and the $\mathbb{R}^2$-fibration to produce a plane bundle over the surface $S$ with Euler number $S\cdot S = 1$. 

The construction of $\mathcal D$ is not complicated: we cut $S$ into an appropriate pleated disc $D^2$, lift it to $\matH^4$, and then expand it orthogonally to a domain $\mathcal D$. The only technical problem is that we are not able to visualise $\matH^4$ and its tessellation into right-angled 120-cells, so many simple geometric sentences like ``these two hyperplanes in $\matH^4$ do not intersect'' have to be verified by analysing $S$ carefully.

\begin{rem}
%
Proposition \ref{pi_1(S):prop} shows in particular the following fact: there is a cocompact arithmetic group $\Gamma \subset {\rm Isom}(\matH^4)$ that contains a geometrically finite surface subgroup $\pi_1(S)$ such that $S$ has genus 3 and $\tilde M = \matH^4/_{\pi_1(S)}$ is a bundle over $S$ with $S\cdot S=1$. Here $\Gamma$ is the reflection group of the Coxeter simplex associated to the right-angled 120-cell, see the beginning of Section \ref{final:section}.

We note that it is possible to deduce Theorem \ref{odd:teo} directly from this fact using a separability argument, without need of the explicit construction of $M$. We thank Alan Reid for pointing this out to us.

The argument goes as follows. 
Since $\Gamma$ is GFERF \cite{BHW}, the geometrically finite subgroup $\pi_1(S)$ is separable in $\Gamma$. By \cite[Lemma 6.3]{KRS}, the closure of $\pi_1(S)$ in the profinite completion $\widehat\Gamma$ is isomorphic to the profinite completion $\widehat{\pi_1(S)}$. Moreover by \cite[Proposition 6.8]{KRS} the group $\widehat{\pi_1(S)}$ is torsion-free, and  by the arguments of \cite[Section 7.1]{KRS} one shows that there is a torsion-free orientation-preserving subgroup $\Gamma'<\Gamma$ of finite index that contains $\pi_1(S)$. By separability, one can assume that $S$ embeds in a finite index cover of the closed manifold $\matH^4/_{\Gamma'}$. 

We note however that the determination of such a $\pi_1(S)$ inside $\Gamma$ is a non-obvious task: in the case described here, we really needed the $Y$-shaped piece, or at least the portion of it which is close to $S$, to construct a surface $S$ with $S\cdot S = 1$. If one could prove that some of the Gromov -- Lawson -- Thurston examples with odd $S\cdot S$ are contained in some arithmetic lattice, then more non-spin arithmetic closed hyperbolic four-manifolds would arise.
\end{rem}

\subsection{Cutting the surface $S$} \label{cutting:subsection}

The surface $S$ lies in the two-skeleton of $M$ and is tessellated into $16$ right-angled pentagons, where $4,6,6$ of them lie in $S_0, S_1, S_2$ respectively. The tessellation is shown in Figure \ref{S_tessellated:fig}: the patient reader may check that there are indeed 16 pentagons in the figure. The 16 edges where the surface $S$ is pleated are thicker in the figure: there are 6 in the interior of $S_1$, 6 in the interior of $S_2$, and 4 in the graph $\Theta$. The one-holed torus $S_0$ is totally geodesic. One checks easily that every vertex contributes with 0 to $S\cdot S$ in the Gromov -- Lawson -- Thurston formula, except the two vertices of $\Theta$, that are drawn in white in the picture, that contribute with $\frac 12$ each. Their link is represented in Figure \ref{sfera_esempi:fig}-(right).

\begin{figure}
 \begin{center}
 \labellist
\small\hair 2pt
\pinlabel $S_2$ at 5 5
\pinlabel $S_0$ at 267 5
\pinlabel $S_1$ at 190 200
\pinlabel $\Theta$ at 175 125
\pinlabel $v$ at 100 130
\endlabellist
  \includegraphics[width = 8 cm]{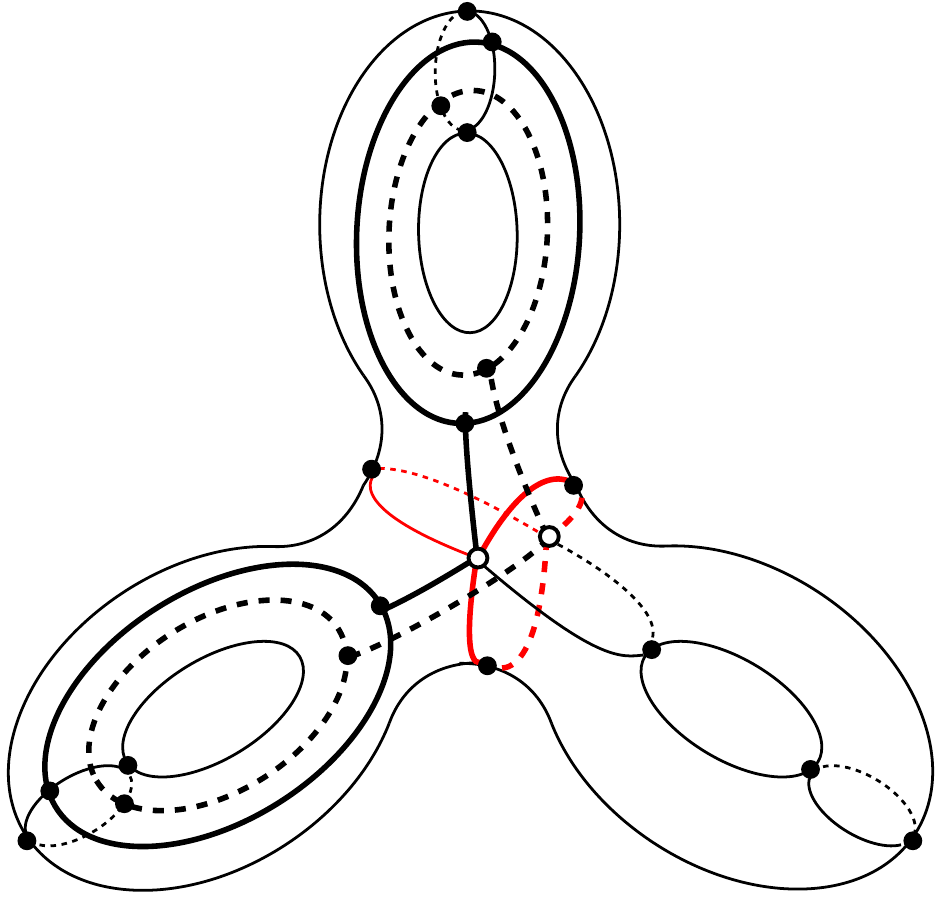}
 \end{center}
 \caption{The surface $S$ is tessellated into $16$ right-angled pentagons. It is pleated along the thick edges and smooth along the thin edges. The two vertices that contribute with $\frac 12$ to the self-intersection $S\cdot S$ are the vertices of $\Theta$ and are drawn in white.}
 \label{S_tessellated:fig}
\end{figure}
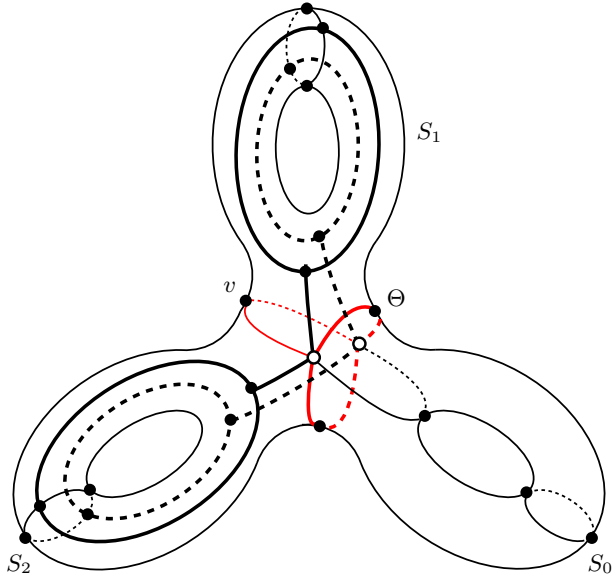

We now cut $S$ open along all the thin edges of Figure \ref{S_tessellated:fig}, except those incident either to the vertex $v$ or to one of the two white vertices. The result is a pleated disc $D^2$ as in Figure \ref{pleated_disk:fig}, tessellated into 16 pentagons and having $v$ at its centre. We lift it to a disc $D^2 \subset\mathbb{H}^4$ contained in the $2$-skeleton of the tessellation of $\mathbb{H}^4$ into copies of the $120$-cell. A crucial fact to note here is that we have obtained $D^2$ by cutting $S$ only along thin (that is, non pleated) edges of $S$.

The boundary of the disc $D^2$ is subdivided into $20$ \emph{sides} as shown in Figure \ref{S_tessellated:fig}, and each side is realised in $\mathbb{H}^4$ as a union of geodesic \emph{arcs}, with each arc corresponding to an edge of a pentagon. Notice that some of the sides are pleated, i.e.\ some of the corresponding geodesic arcs make right angles at their common endpoint. The 20 sides and the 16 pentagons in the figure are labeled with some letters $A$, $B$, $C$, $D$, $E$, $F$ and $P_0, P_1, P_2, P_3$ respectively: the reason for this marking will be explained soon.

\begin{figure}[htbp]
 \begin{center}
  \labellist
\small\hair 2pt
\pinlabel $v$ at 545 440
\endlabellist
  \includegraphics[width = 9 cm]{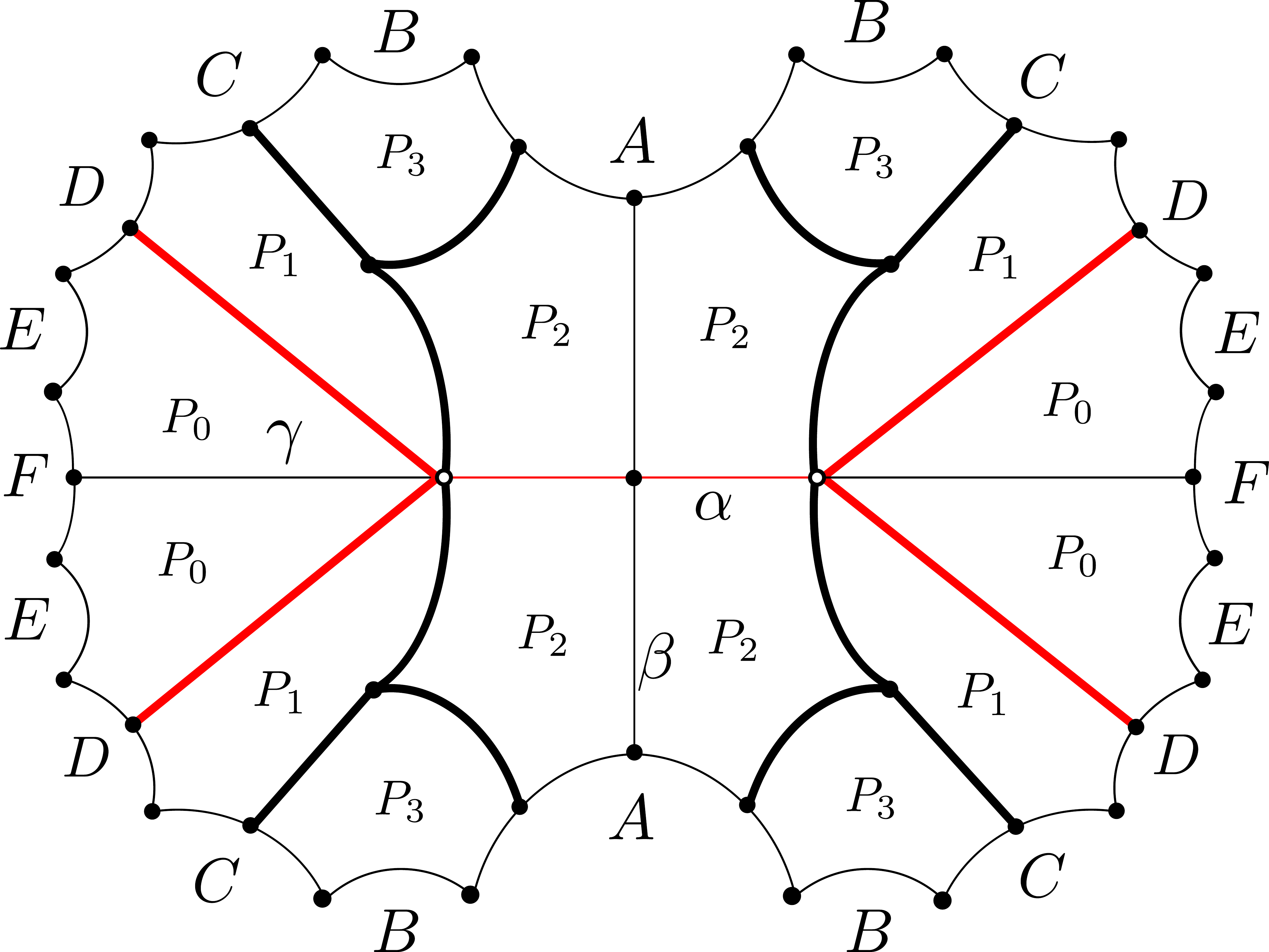}
 \end{center}
 \caption{The pleated disc $D^2$ with its tessellation into $16$ right-angled pentagons. It is pleated at right angles along the thick edges and smooth along the thin edges. The red edges correspond to the graph $\Theta$.}
 \label{pleated_disk:fig}
\end{figure}

\subsection{The fundamental domain $\mathcal D$}

Now, to each side $s$ in the boundary of $D^2$ we wish to associate a hyperplane $H_s$ in $\mathbb{H}^4$. We proceed in the following way. Consider a pentagon $P\subset D^2$ which intersects $s$ in one of its edges. There is a unique hyperplane $H_s$ in $\mathbb{H}^4$ which contains $s$ and intersects $P$ orthogonally along $s$. Notice that the pentagon $P$ is not uniquely determined, as some sides of $D^2$ intersect more than one pentagon. However, the resulting hyperplane $H_s$ does not depend on the choice of $P$.

Consider a hyperplane $H_s$ constructed as above. Its intersection with the pentagon $P$ is a geodesic arc, with $P$ lying in one of the two halfspaces determined by $H_s$. Let us call $\mathcal{H}_s$ such halfspace. We define the set $\mathcal{D}\subset \mathbb{H}^4$ as the intersection of the halfspaces of the form $\mathcal{H}_s$, where $s$ varies over the sides of the disc $D^2$:
$$\mathcal{D}=\bigcap_{s}\mathcal{H}_s.$$

Consider now two hyperplanes $H_s$ and $H_{s'}$, corresponding to adjacent sides $s, s'$ in the boundary of $D^2$. Clearly these two hyperplanes intersect along a hyperbolic plane, that contains the common vertex of $s$ and $s'$ and is orthogonal to the adjacent pentagon in $D^2$. We claim that these are the \emph{only} intersections between the hyperplanes $H_s$:

\begin{claim} \label{claim}
Suppose that $s$ and $s'$ are non-adjacent sides of the disc $D^2$. Then the corresponding hyperplanes $H_s$ and $H_{s'}$ do not intersect in $\mathbb{H}^4$. 
\end{claim}

By construction, the hyperplanes $H_s$ are hyperplanes in the tessellation of $\mathbb{H}^4$ into copies of the $120$-cell. Before proving Claim \ref{claim}, we take a closer look at the combinatorial properties of this tessellation.

\subsection{The $120$-cell tessellation of $\mathbb{H}^4$ is naturally coloured} \label{sec:tessellation}
Consider a right-angled hyperbolic $120$-cell $C \subset \mathbb{H}^4$. By reflecting it along its facets we produce a tessellation of $\mathbb{H}^4$ into copies of $C$. Now, consider a $k$-dimensional face $F$ of this tessellation and let $H$ be the $k$-dimensional subspace in $\mathbb{H}^4$ which contains it. The face $F$ is obtained by applying a number of reflections to some \emph{unique} $k$-dimensional face $F_0$ of $C$, and any other face $F'\subset H$ of the tessellation is obtained in the same way from the same face $F_0$ of $C$. Therefore, it makes sense to label the whole subspace $H$ with the face $F_0$ of $C$. We have therefore defined a ``labeling'' function from the set of $k$-dimensional subspaces of the tessellation to the set of $k$-dimensional faces of $C$. 

Now, consider two hyperplanes $H_1$ and $H_2$ of the tessellation. A necessary condition for $H_1$ and $H_2$ to intersect is that their labels should correspond to a pair of adjacent facets of the $120$-cell $C$. Conversely, if their labels correspond to non-adjacent facets of $C$, the hyperplanes cannot intersect. The intersection patterns for the facets of $C$ can be visualised much more easily by considering the dual polytope to $C$, the $600$-cell $C^*$.
This polytope has $600$ tetrahedral facets, $1200$ triangular faces, $720$ edges and $120$ vertices, and its boundary is a simplicial complex homeomorphic to the $3$-sphere. The correspondences between the strata of the two polytopes is as follows:
\begin{itemize}
\item \{Dodecahedra of $C$\} $\longleftrightarrow$ \{Vertices of $C^*$\}
\item \{Pentagons of $C$\} $\longleftrightarrow$ \{Edges of $C^*$\}
\item \{Edges of $C$\} $\longleftrightarrow$ \{Triangles of $C^*$\}
\item \{Vertices of $C$\} $\longleftrightarrow$ \{Tetrahedra of $C^*$\}
\end{itemize}
Clearly, two dodecahedral facets of $C$ intersect if and only if the corresponding vertices of $C^*$ are joined by an edge. Now, consider a pentagon $P$ in $C$, corresponding to an edge $e$ of $C^*$. There are $5$ distinct tetrahedra in $C^*$ which have $e$ as an edge, as shown in Figure \ref{pentagon:fig}. There are exactly $5$ vertices in these tetrahedra which are not vertices of $e$, and these correspond to the dodecahedral facets of $C$ which intersect $P$ orthogonally in an edge.

\begin{figure} 
 \begin{center}
  \includegraphics[width = 5 cm]{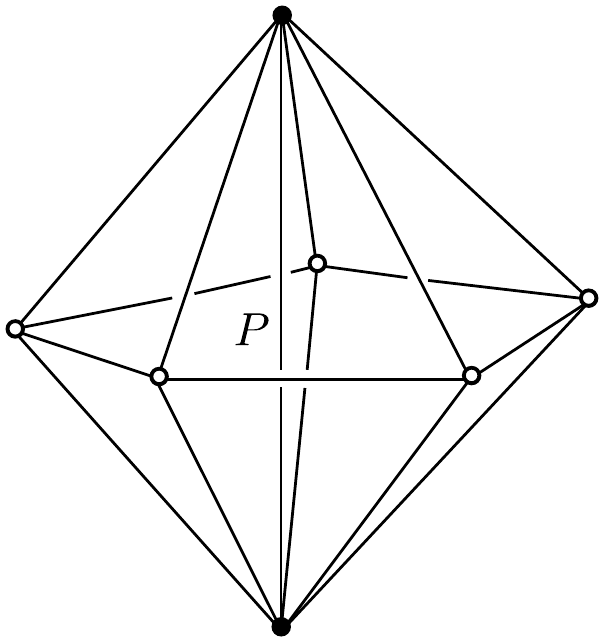}
 \end{center}
 \caption{The five simplices adjacent to an edge of the $600$-cell $C^*$. If the central edge corresponds to a pentagon $P$ of the $120$-cell $C$, the five white vertices correspond to the dodecahedral facets of $C$ which intersect $P$ in one of its edges.}
 \label{pentagon:fig}
\end{figure}

\subsection{Proof of Claim \ref{claim}} \label{sec:claim_proof}
Consider the pleated disc $D^2$ of Figure \ref{pleated_disk:fig}. By the discussion above, every pentagon of $D^2$ is labeled by some pentagon of $C$, and every side $s$ of $D^2$ is labeled by the facet of $C$ that is assigned to the corresponding hyperplane $H_s$.

With some patience one discovers that the pentagons are marked with only $4$ distinct labels $P_0, P_1, P_2,$ and $P_3$, as shown in Figure \ref{pleated_disk:fig}. The $4$ pentagons with label $P_0$ are those lying in the totally geodesic one-holed torus $S_0$. The remaining pentagons have labels $P_1,P_2$ and $P_3$, and those in the upper (resp.\ lower) half of the picture lie in $S_1$ (resp.\ $S_2$). A careful analysis shows that the sides of $D^2$ are marked with $6$ different labels $A,B,C,D,E,F$ as shown in Figure \ref{pleated_disk:fig}. By dualising the $120$-cell, we associate to the $4$ pentagons $4$ distinct edges and to the boundary hyperplanes $6$ distinct vertices of the $600$-cell as in Figure \ref{grafo_600cella:fig}. 
\begin{figure}[htbp]
 \begin{center}
  \includegraphics[width = 7.5 cm]{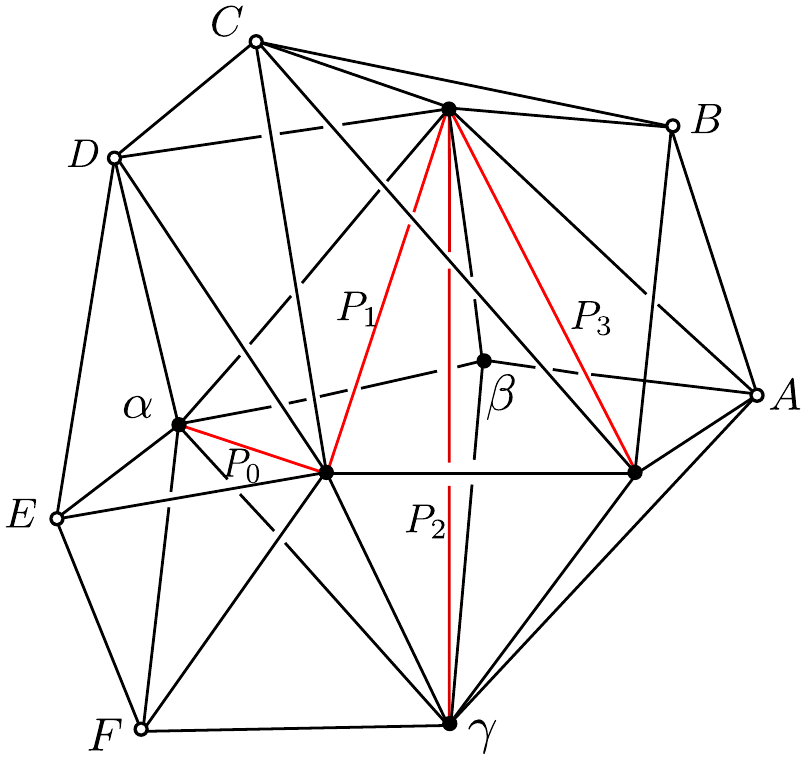}
 \end{center}
 \caption{Labels for the pentagons and sides of $D^2$, seen dually in the $600$-cell $C^*$. The pentagons are drawn as red edges, while the hyperplanes corresponding to the sides are drawn as white vertices. The figure shows only the portion of 600-cell that is of interest for us.}
 \label{grafo_600cella:fig}
\end{figure}

Note that there is an edge connecting the vertex with label $A$ to the vertex with label $B$, as one would expect by noticing that there are adjacent sides of $D^2$ with labels $A$ and $B$. More importantly, we point out that there are \emph{no extra edges} in the $600$-cell connecting any pair of the vertices $A$, $B$, $C$, $D$, $E$, and $F$ apart from those shown in the figure. This means that we can, for instance, prove that $H_s \cap H_{s'} = \emptyset$ if the two edges $s,s'$ have labels $A$ and $C$, or $A$ and $D$, and so on.

This excludes many possible unwanted intersections, but not all. For example, a hyperplane $H_s$ with label $B$ associated to a side $s$ of $D^2$ could intersect a hyperplane $H_{s'}$ with label $C$ associated to a side $s'$ of $D^2$, with $s$ not adjacent to $s'$. In order to exclude this type of intersection we proceed as follows. 
Consider the internal edges of $D^2$ with labels $\alpha$, $\beta$ and $\gamma$ as shown in Figure \ref{pleated_disk:fig}. Note that these edges are not pleated, therefore there are $3$ hyperplanes in $\mathbb{H}^4$, each containing one these edges and orthogonal to the disc $D^2$. By a slight abuse of notation, we label these $3$ hyperplanes by $\alpha$, $\beta$ and $\gamma$ respectively. They correspond to $3$ vertices of the $600$-cell, as shown in Figure \ref{grafo_600cella:fig}. Each of these hyperplanes separates $\mathbb{H}^4$ into two halfspaces. Now, for every possible unwanted intersection between hyperplanes $H_s$ and $H_{s'}$ with adjacent labels (but with non-adjacent sides $s$ and $s'$), we can always find at least one hyperplane with label $\alpha$, $\beta$ or $\gamma$ that separates $H_s$ and $H_{s'}$, i.\ e.\ such that $H_s$ and $H_{s'}$ lie in opposite halfspaces with respect to the chosen hyperplane. Therefore $H_s$ and $H_{s'}$ turn out to be disjoint. 

For example, consider  in Figure \ref{pleated_disk:fig} the upper left side with label $B$ and the lower right side with label $C$. The two corresponding hyperplanes
are separated by any of the
three hyperplanes $\alpha$, $\beta$ and $\gamma$. Similarly, consider the two hyperplanes with label $F$. They
are separated by
the hyperplane with label $\beta$. By repeating this reasoning for all possible pairs of non-adjacent sides with adjacent labels, we conclude that there are no unwanted intersections between the hyperplanes $H_s$, and Claim \ref{claim} is proven.

\subsection{Conclusion of the proof of Proposition \ref{pi_1(S):prop}}

First, we notice that the interior of the pleated disc $D^2$ is entirely contained in the interior of the domain $\mathcal{D}$. This follows from the fact that none of the internal edges of the tessellation of $D^2$ into pentagons is contained in a bounding hyperplane $H_s$ of $\mathcal{D}$, and therefore $D^2$ cannot intersect the bounding hyperplanes of $\mathcal{D}$ in its interior. This can be verified by noticing that none of the triangular faces of the $600$-cell corresponding to the internal edges of $D^2$ has vertices with label $A,B,C,D,E$ or $F$.

Following Kuiper's terminology \cite[Section 3.1]{K}, the polyhedron $\mathcal D$ is a right-angled 4-dimensional \emph{convex 20-gon}. It has 20 cyclically consecutive facets; each facet contains one side of $D^2$ and is isometric to $\matH^3$ minus two open half-spaces with disjoint closures in $\overline{\matH}^3$. Two consecutive facets are incident at a right angle along a copy of $\matH^2$. This is a consequence of Claim \ref{claim}. 

We now split each side labeled with $A$ in Figure \ref{pleated_disk:fig} into two sides (that we still label with the letter $A$), by cutting it at its middle point. The number of sides of $D^2$ grows from 20 to 22. We also split the corresponding facets of $\mathcal D$ into two facets (along the plane orthogonal to the middle point of the original $A$), that we now think as meeting with a dihedral angle $\pi$. Now the domain $\mathcal D$ is a 22-gon, with consecutive facets meeting either at $\frac \pi 2$ or $\pi$ angle.

By construction $\pi_1(S)$ acts isometrically on $\matH^4$ and by examining Figures \ref{S_tessellated:fig} and \ref{pleated_disk:fig} we check that the action is generated by some pairing on the 22 sides of $D^2$ that give rise to $S$. Every side with label $B,D,E$, or $F$ is paired to a side with the same letter, while the 4 sides labeled by $C$ are paired with the 4 sides labeled by $A$.

Since all the sides in $\partial D^2$ are made of thin (that is, non pleated) edges, the isometry in $\pi_1(S)$ that sends a side $s$ to some side $s'$ also sends isometrically the hyperplane $H_s$ to $H_{s'}$. Therefore it pairs isometrically the corresponding facets of $\mathcal D$. It is crucial here that $\partial D^2$ is made of thin edges.

Summing up, the action of $\pi_1(S)$ on $\matH^4$ is generated by some face pairings of $\mathcal{D}$. 
By Poincar\'e's Fundamental Polyhedron Theorem, the action is faithful and $\mathcal{D}$ is a fundamental domain, see \cite[Theorem 4.14]{EP}.  
Moreover, since $\mathcal D$ is finite-sided, the Kleinian group $\pi_1(S)$ is geometrically finite.

Finally, being a convex $22$-gon, the domain $\mathcal{D}$ is homeomorphic to $D^2\times\mathbb{R}^2$, with $D^2$ itself embedded as $D^2\times\{(0,0)\}$. 
The $\mathbb{R}^2$-fibration can be adjusted to be preserved by the pairing maps and everything can be smoothened, so the quotient 
$\tilde M=\mathbb{H}^4/_{\pi_1(S)}$ is diffeomorphic to a rank-$2$ real vector bundle over $S$ with Euler number $S\cdot S = 1$.

\section{Proof of Theorem \ref{main:teo}} \label{final:section}
We now prove Theorem \ref{main:teo}.
We have built in Theorem \ref{odd:teo} a non-spin hyperbolic $4$-manifold $M$ which is tessellated into copies of the right-angled $120$-cell $C$. Since $C$ is a regular polytope, the manifold $M$ is an orbifold covering of the characteristic simplex $\Delta\cong C/_{\mathrm{Isom}(C)}$ of $C$.

Let $\Gamma<\mathrm{Isom}(\mathbb H^4)$ be the Coxeter group generated by reflections in the facets of $\Delta$. By \cite{Bug} (see also \cite{VS}), $\Gamma$ is arithmetic of simplest type, and the associated admissible quadratic form $\mathrm{f}$ of signature $(4,1)$ is defined over the field $k=\mathbb{Q}[\sqrt{5}]$. More specifically, $\Gamma$ is a subgroup of the group $\mathrm{O}(\mathrm{f},R_k)$, where $R_k$ is the ring of integers of the field $k$.

We will apply the following result from \cite{KRS}:

\begin{lemma}\label{lemma:embedding}
Let $M^n$ be an orientable, arithmetic hyperbolic $n$-manifold of simplest type, with associated quadratic form $\mathrm{f}$ defined over a field $k$. Suppose that the group $\pi_1(M^n)<\Or(f)$ is contained in the subgroup of $k$-points $\mathrm{O}(\mathrm{f},k)$. Then $M^n$ geodesically embeds in an orientable hyperbolic $(n+1)$-manifold $M^{n+1}$ which is itself arithmetic of simplest type, with associated form $\mathrm{g}$ defined over the same field $k$. Moreover $\pi_1(M^{n+1})<\mathrm{O}(\mathrm{g},k)$. If $M^n$ is compact and defined over a proper extension of $\mathbb{Q}$, so is $M^{n+1}$.
\end{lemma}

\begin{proof}[Sketch of proof]
Choose the form $\mathrm{g}=y^2+\mathrm{f}$, where $y$ denotes a new coordinate. Notice that $\mathrm{g}$ has signature $(n+1,1)$ and is admissible over $k$ because so is $\mathrm{f}$. By \cite[Proposition 2.1]{KRS}, a torsion-free arithmetic lattice $\Gamma=\pi_1(M^n)<\mathrm{O}(\mathrm{f},k)$ injects in an arithmetic lattice $\Lambda<\mathrm{O}(\mathrm{g},k)<\mathrm{Isom}(\mathbb{H}^{n+1})$. Moreover the group $\Gamma$ is geometrically finite and therefore separable in $\Lambda$ by \cite{BHW}. Separabilty of $\Gamma$ allows us to find a torsion-free, finite index subgroup $\Lambda'<\Lambda$ which contains $\Gamma$, and such that $M^n$ embeds geodesically in $M^{n+1}=\mathbb{H}^{n+1}/_{\Lambda'}$. 
Finally, note that $M^n$ and $M^{n+1}$ are defined over the same field $k$. By \cite[Proposition 6.4.4]{Morris}, if $k$ is a proper extension of $\mathbb{Q}$, then both $M^n$ and $M^{n+1}$ are compact. 
\end{proof}

The hypothesis of Lemma \ref{lemma:embedding} hold in particular for $\pi_1(M)<\Gamma<\mathrm{O}(\mathrm{f},R_k)<\mathrm{O}(\mathrm{f},k)$.
We now build a sequence of $n$-dimensional manifolds $M^n$, $n\geq 4$, by choosing $M^4=M$ and repeatedly applying Lemma \ref{lemma:embedding} so that each $M^n$ embeds as a totally geodesic submanifold in $M^{n+1}$. 

Each $M^n$ is not spin by a standard argument: the manifold $M^4$ is not spin, hence $w_2(M^4)\neq 0$, and $M^n \subset M^{n+1}$ has a trivial normal bundle (since they are both orientable and the codimension is 1), so by the natural properties of the Stiefel-Whitney classes $w_2(M^n) \neq 0$ implies $w_2(M^{n+1})\neq 0$. 

More specifically, we have
$$w(TM^{n+1}|_{M^n})=w(TM^n)\smile w(\nu M^n) = w(TM^n).$$ 
If $w_2(TM^n) \neq 0$, then $w_2(TM^{n+1}\big|_{M^n}) \neq 0$ and by naturality of Stiefel-Whitney classes we also get $w_2(TM^{n+1})\neq 0$.

\end{document}